\documentclass[reqno]{amsart}
\usepackage{amsmath, amssymb, amsfonts, mathrsfs, url, hyperref, dsfont}

\newtheorem{theorem}{Theorem}
\newtheorem*{theorem*}{Theorem}
\newtheorem*{conjecture*}{Conjecture}
\newtheorem{lemma}{Lemma}
\theoremstyle{definition}
\newtheorem{definition}{Definition}
\theoremstyle{remark}

\title{Inhomogeneous Khintchine--Groshev theorem without monotonicity}
\author{Seongmin Kim}
\address{Department of Mathematical Sciences, Seoul National University, 1, Gwanak-ro, Gwanak-gu, Seoul, 08826, Republic of Korea}
\email{seongmin.kim@snu.ac.kr}

\begin{document}

\begin{abstract}
   The Khintchine--Groshev theorem in Diophantine approximation theory says that there is a dichotomy of the Lebesgue measure of sets of $\psi$-approximable numbers, given a monotonic function $\psi$. Allen and Ram\'irez removed the monotonicity condition from the inhomogeneous Khintchine--Groshev theorem for cases with $nm\geq3$ and conjectured that it also holds for $nm=2$. In this paper, we prove this conjecture in the case of $(n,m)=(2,1)$. We also prove it for the case of $(n,m)=(1,2)$ with a rational inhomogeneous parameter.
\end{abstract}


\maketitle

\section{Introduction}
One of the fundamental results in classical Diophantine approximation is Khintchine's theorem \cite{Khi24}. This theorem states that for any monotonic function $\psi: \mathbb{N} \to \mathbb{R}_{\geq 0}$, the Lebesgue measure of the set of $\psi$-approximable numbers, defined by
$$\mathscr{A}_{1,1}(\psi) = \{x \in [0,1) : |qx - p| < \psi(q) \text{ for infinitely many } (q,p) \in \mathbb{N} \times \mathbb{Z}\}$$
is either 0 or 1, depending on whether the series $\sum_{q\in \mathbb{N}} \psi(q)$ converges or diverges. Khintchine later extended this result to the simultaneous approximation case \cite{Khi26}, and Groshev further generalized it to systems of linear forms \cite{Gro38}. In this paper, we focus on the inhomogeneous Khintchine--Groshev theorem, which was established by Sz\"usz \cite{Szu58} and Schmidt \cite{Sch64}.

\subsection{Inhomogeneous Khintchine--Groshev theorem}
For positive integers \\ $n,m \in \mathbb{N}$, a function $\psi : \mathbb{N} \to \mathbb{R}_{\geq 0}$, and an inhomogeneous parameter $\mathbf{y} \in \mathbb{R}^m$, let $\mathscr{A}^\mathbf{y}_{n,m}(\psi)$ denote the set of $\mathbf{x} \in [0,1)^{nm}$ such that
$$|\mathbf{qx} - \mathbf{p} - \mathbf{y}| < \psi(|\mathbf{q}|) \text{ for infinitely many } (\mathbf{q}, \mathbf{p}) \in (\mathbb{Z}^n \setminus \{\mathbf{0}\}) \times \mathbb{Z}^m.$$
Here $\mathbf{q} \in \mathbb{Z}^n \setminus \{\mathbf{0}\}$ and $\mathbf{p} \in \mathbb{Z}^m$ are considered as row vectors, while $\mathbf{x} \in [0,1)^{nm}$ represents an $n \times m$ matrix. The notation $|\cdot|$ refers to the maximum norm, defined as $|(q_1, \cdots, q_n)| = \max_{1\leq i \leq n}|q_i|.$ We also denote the Lebesgue measure by $|\cdot|$. The inhomogeneous Khintchine--Groshev theorem is as follows:
\begin{theorem*}[Inhomogeneous Khintchine--Groshev]
    Let $n,m \in \mathbb{N}$, $\psi : \mathbb{N} \to \mathbb{R}_{\geq 0}$, and $\mathbf{y} \in \mathbb{R}^m$. Then
    $$|\mathscr{A}_{n,m}^\mathbf{y} (\psi)| = \begin{cases}  0  \quad \text{ if } \quad \sum_{q=1}^\infty q^{n-1} \psi(q)^m < \infty, \\ 1 \quad \text{ if } \quad \sum_{q=1}^\infty q^{n-1} \psi(q)^m = \infty, \text{ and }\psi \text{ is monotonic.}    \end{cases}$$
\end{theorem*}

Monotonicity is necessary in the case $(n,m) =(1,1)$, as shown by the counterexample of Duffin and Schaeffer \cite{DS41} in the homogeneous setting $(y=0)$, and by Ram\'irez \cite{Ram17} in the inhomogeneous setting ($y \in \mathbb{R}$). After a long history of development, Beresnevich and Velani \cite{BV09} removed the monotonicity condition for the homogeneous case when $nm \geq 2$. In the inhomogeneous case, Allen and Ram\'irez \cite{AR22} removed the monotonicity condition for $nm\geq 3$, and they conjectured that it can also be removed for the case $nm =2$ (see \cite{Gal65}, \cite{Spr79}, \cite{Yu21}, and \cite{AR22} for historical details).
\begin{conjecture*}[\cite{AR22}, Conjecture 1]
    Let $(n,m) = (1,2)$ or $(2,1)$, $\psi: \mathbb{N} \to \mathbb{R}_{\geq 0}$, and $\mathbf{y} \in \mathbb{R}^m$. Then
    $$|\mathscr{A}_{n,m}^\mathbf{y} (\psi)| = \begin{cases}  0  \quad \text{ if } \quad \sum_{q=1}^\infty q^{n-1} \psi(q)^m < \infty, \\ 1 \quad \text{ if } \quad \sum_{q=1}^\infty q^{n-1} \psi(q)^m = \infty. \end{cases}$$
\end{conjecture*}

Hauke \cite{Hau23} proved this conjecture in the case $(n,m) = (2,1)$ for any non-Liouville irrational inhomogeneous parameter $y$. In this paper, we prove the conjecture for $(n,m)=(2,1)$ with any inhomogeneous parameter.

\begin{theorem}\label{thm1}
    Let $(n,m) = (2,1),\ \psi: \mathbb{N} \to \mathbb{R}_{\geq 0}$ and $y \in \mathbb{R}$. Then
    $$|\mathscr{A}_{2,1}^y (\psi)| = \begin{cases}  0  \quad \text{ if } \quad \sum_{q=1}^\infty q \psi(q) < \infty, \\ 1 \quad \text{ if } \quad \sum_{q=1}^\infty q \psi(q) = \infty. \end{cases}$$
\end{theorem}
We also prove that, for a rational inhomogeneous parameter $\mathbf{y} \in \mathbb{Q}^m$, the monotonicity condition can be removed in the case $(n,m) = (1,2)$.
\begin{theorem}\label{thm2}
    Let $nm \geq 2,\ \psi: \mathbb{N} \to \mathbb{R}_{\geq 0}$ and $\mathbf{y} \in \mathbb{Q}^m.$ Then
    $$|\mathscr{A}_{n,m}^\mathbf{y} (\psi)| = \begin{cases}  0  \quad \text{ if } \quad \sum_{q=1}^\infty q^{n-1} \psi(q)^m < \infty, \\ 1 \quad \text{ if } \quad \sum_{q=1}^\infty q^{n-1} \psi(q)^m = \infty. \end{cases}$$
\end{theorem}

\vspace{.1cm}
\subsection{Strategy of proofs} Let $n,m \in \mathbb{N}$, $\psi : \mathbb{N} \to \mathbb{R}_{\geq 0}$ and $\mathbf{y} \in \mathbb{R}^m$ be given. We may assume that $\psi : \mathbb{N} \to [0,\frac{1}{2})$ and $\mathbf{y} \in [0,1)^m$. For each $\mathbf{q} \in \mathbb{Z}^n \setminus \{\mathbf{0}\}$, define
$$A_\mathbf{q} := \{\mathbf x \in [0,1)^{nm} : |\mathbf{qx} - \mathbf{p} - \mathbf{y}| < \psi(|\mathbf{q}|) \text{ for some } \mathbf{p} \in \mathbb{Z}^m\}.$$
By definition of $\mathscr{A}_{n,m}^\mathbf{y} (\psi)$, we have
$$\mathscr{A}_{n,m}^\mathbf{y} (\psi) = \limsup_{|\mathbf{q}| \to \infty} A_\mathbf{q} = \bigcap_{Q=1}^\infty \bigcup_{q\geq Q} \bigcup_{|\mathbf{q}| = q} A_\mathbf{q} .$$
Since the map $\mathbf{x} \mapsto \mathbf{qx}$ (mod 1) is measure-preserving (\cite{Spr79}, Chap 1, Lemma 8), and $A_\mathbf{q}$ is the preimage under this map of a ball (under the maximum norm) centered at $\mathbf{y}$ with radius $\psi(|\mathbf{q}|)$ in $[0,1)^{m} \cong \mathbb{R}^m/\mathbb{Z}^m$, it follows that
$$|A_\mathbf{q}| = (2\psi(|\mathbf{q}|))^m.$$
The number of vectors $\mathbf{q}$ with norm $|\mathbf{q}| = q$ is comparable to $q^{n-1}$. Therefore, by the Borel--Cantelli lemma, the convergence part of the theorem follows immediately.

The proof of the divergence part begins with the quasi-independence on average (QIA). A sequence $(B_\mathbf{q})_{\mathbf{q} \in \mathbb{Z}^n \setminus \{ \mathbf{0}\}}$ of measurable subsets of a finite measure space $(X,\mu)$ is \emph{quasi-indendent on average} if
$$\limsup_{Q \to \infty} \frac{\left(\sum_{|\mathbf{q}|=1}^Q \mu(B_\mathbf{q})\right)^2}{\sum_{|\mathbf{q}|,|\mathbf{r}|=1}^Q \mu(B_\mathbf{q} \cap B_\mathbf{r})}>0.$$
Instead of the second Borel--Cantelli lemma, which ensures full measure under independence assumptions, the divergence Borel--Cantelli lemma, which originates from the Chung--Erd\H{o}s lemma, provides the lower bound (see \cite[Chap 1, Lemma 5]{Spr79})
$$\mu(\limsup_{|\mathbf{q}| \to \infty}B_\mathbf{q}) \geq \limsup_{Q \to \infty} \frac{\left(\sum_{|\mathbf{q}|=1}^Q \mu(B_\mathbf{q})\right)^2}{\sum_{|\mathbf{q}|,|\mathbf{r}|=1}^Q \mu(B_\mathbf{q} \cap B_\mathbf{r})}.$$
Allen and Ram\'{i}rez \cite{AR22} proved the divergence part for $nm\geq 3$ by showing that $(A_\mathbf{q})_{\mathbf{q} \in \mathbb{Z}^n \setminus \{\mathbf{0}\}}$ is QIA and QIA gives full measure. However, we cannot guarantee that $(A_\mathbf{q})_{\mathbf{q} \in \mathbb{Z}^n \setminus \{\mathbf{0}\}}$ is QIA in $nm=2$ case. Thus for the homogeneous case, i.e. for $\mathbf{y} = \mathbf{0}$, Beresnevich and Velani \cite{BV09} constructed a sequence $(A'_\mathbf{q})_{\mathbf{q} \in \mathbb{Z}^n \setminus \{\mathbf{0}\}}$ of subsets of $A_\mathbf{q}$ by
$$A'_\mathbf{q} = \{\mathbf{x} \in [0,1)^{nm} : |\mathbf{qx} - \mathbf{p}|< \psi(|\mathbf{q}|) \text{ for some } \mathbf{p} \in \mathbb{Z}^m \text{ with } \gcd(\mathbf{q},\mathbf{p})=1\},$$
where $\gcd(\mathbf{q}, \mathbf{p}) := \gcd(q_1, \cdots, q_n, p_1, \cdots, p_m).$
They \cite{BV09} showed that $(A'_\mathbf{q})_{\mathbf{q} \in \mathbb{Z}^n \setminus \{\mathbf{0}\}}$ is QIA and applied the zero-one law from \cite{BV08} to conclude the result.

For the proof of Theorem \ref{thm1}, let $(n,m) = (2,1)$ and let $y \in \mathbb{R}$ be an inhomogeneous parameter. By Dirichlet's theorem, for each $d \in \mathbb{Z}\setminus\{0\},$ there exist $a_d \in \mathbb{Z}$ and $b_d \in \mathbb{N}$ such that
\begin{equation}\label{a,b}
    |b_d y - a_d| < \frac{1}{|d|}, \quad 1\leq b_d \leq |d|, \quad \gcd(a_d, b_d) = 1.
\end{equation}
For each $\mathbf{q} \in \mathbb{Z}^2 \setminus \{ \mathbf{0} \}$, denote $d= \gcd(\mathbf{q})$. This defines the sequence
\begin{equation}\label{a,b,q}
    (a_\mathbf{q}, b_\mathbf{q}) := (a_{\gcd(\mathbf{q})}, b_{\gcd(\mathbf{q})}) = (a_d, b_d).
\end{equation}
Now, we define $\tilde{A}_\mathbf{q}$ to be the set of $\mathbf{x} \in [0,1)^2$ such that
$$|\mathbf{qx} - p - y| < \psi(|\mathbf{q}|) \text{ for some } p \in \mathbb{Z} \text{ with } \gcd(\mathbf{q}, b_\mathbf{q}p+a_\mathbf{q}) = 1.$$
We prove that $(\tilde{A}_\mathbf{q})_{\mathbf{q} \in \mathbb{Z}^2 \setminus \{ \mathbf{0}\}}$ is QIA (Theorem \ref{thm3}) using techniques from \cite{Hau23}, and being QIA is enough to show that $\tilde{\mathscr{A}}^y_{2,1}(\psi) := \limsup_{|\mathbf{q}| \to \infty} \tilde{A}_\mathbf{q}$ has full measure (Theorem \ref{thm4}), thus completing the proof of Theorem \ref{thm1}. Notably, when $y$ is a non-Liouville irrational number, Hauke \cite{Hau23} established the QIA of $(A_\mathbf{q})_{\mathbf{q} \in \mathbb{Z}^2 \setminus \{\mathbf{0}\}}$ by proving that $|A_\mathbf{q} \cap A_\mathbf{r}|=0$ for many pairs $(\mathbf{q}, \mathbf{r})$. Adapting this approach to handle any $y \in \mathbb{R}$, we show in Lemma \ref{mainlemma} that $|\tilde{A}_\mathbf{q} \cap \tilde{A}_\mathbf{r}|=0$ for many pairs $(\mathbf{q}, \mathbf{r})$, which is the key step in establishing the QIA of $(\tilde{A}_\mathbf{q})_{\mathbf{q} \in \mathbb{Z}^2 \setminus \{ \mathbf{0}\}}$.

For the proof of Theorem \ref{thm2}, we generalize the strategy of \cite{BV08} and \cite{BV09} to cases where the inhomogeneous parameter $\mathbf{y}$ is a rational vector. Let $\mathbf{y} = \frac{\mathbf{a}}{b} = \frac{(a_1, \cdots, a_m)}{b}$ for $a_i \in \mathbb{Z}$ and $b \in \mathbb{N}$ such that $\gcd(\mathbf{a},b) = \gcd(a_1, \cdots, a_n, b) = 1$. We then define $A''_\mathbf{q}$ as the set of $\mathbf{x} \in [0,1)^{nm}$ satisfying
$$|\mathbf{qx} - \mathbf{p} - \mathbf{y}|< \psi(|\mathbf{q}|) \text{ for some } \mathbf{p} \in \mathbb{Z}^m \text{ with } \gcd(\mathbf{q},b\mathbf{p} + \mathbf{a})=1.$$
We will show that $\mathscr{A}_{n,m}^{\mathbf{y}}(\psi)$ satisfies the zero-one law (Theorem \ref{thm6}). Additionally, when $nm \geq 2$, the sequence $(A''_\mathbf{q})_{\mathbf{q} \in \mathbb{Z}^n \setminus \{\mathbf{0}\}}$ is QIA in the divergence case (Theorem \ref{thm5}), thus completing the proof of Theorem \ref{thm2}.

\subsection*{Remarks}
What is known and what remains open for Khintchine--Groshev-type problems is organized in Table 1 of \cite{HR24}. Theorem \ref{thm1} and Theorem \ref{thm2} address two of the cells in this table. The case of $(n,m)=(1,2)$ with an irrational inhomogeneous parameter remains unsolved. Additionally, the moving target cases \cite[Conjecture 2]{AR22}, where $\mathbf{y_\mathbf{q}} \in \mathbb{R}^m$ can vary as $\mathbf{q}$ changes, also remain open for $nm=2$.

Felipe Ramírez suggested that Theorem \ref{thm2} could be extended to a multivariate setting when $m\geq2$ with $\mathbf{y} \in \mathbb{Q}^m$. This generalization follows from \cite[Theorem 2.3]{Ram24}, where the $n=1$, $m \geq 2$ case of Theorem \ref{thm2} can substitute for the concluding step in the proof of \cite[Theorem 2.3]{Ram24}. This result addresses another cell in Table 1 of \cite{HR24}.

The condition $\gcd(\mathbf{q}, b\mathbf{p} + \mathbf{a}) = 1$ in $\tilde{A}_\mathbf{q}$ or $A''_\mathbf{q}$ appears in various references, such as \cite{Sch64}, and more recently in \cite{BHV24} and \cite{CT24}, as noted by Manuel Hauke in correspondence. Notably, the weak inhomogeneous Duffin--Schaeffer conjecture with a rational inhomogeneous parameter was proved in \cite{BHV24}, which parallels Theorem \ref{thm2}.

\subsection*{Acknowledgements}
The author would like to thank Seonhee Lim for introducing the author to this problem and for her guidance throughout this project as well as Felipe Ramírez, Manuel Hauke and the anonymous reviewer for their valuable comments. The author is supported by National Research Foundation of Korea, under Project number NRF005150821G0002223.

\vspace{.1cm}
\section{Preliminaries}
In this section, we gather definitions and lemmas, which are slight modifications of previous results. (See \cite{Spr79}, \cite{BV09}, \cite{AR22} and the references therein.) The following is a generalized version of the notions introduced in the previous section.
\begin{definition}\label{def}
    Let $n,m \in \mathbb{N}$, and let $(\frac{\mathbf{a}_\mathbf{q}}{b_\mathbf{q}})_{\mathbf{q} \in \mathbb{Z}^n \setminus \{\mathbf{0}\}}$ be a sequence of rational vectors in $\mathbb{Q}^m$ such that $\mathbf{a}_\mathbf{q} \in \mathbb{Z}^m,\ b_\mathbf{q} \in \mathbb{N}$ and $\gcd(\mathbf{a}_\mathbf{q}, b_\mathbf{q}) = 1.$ Also, let $(\mathbf{y}_\mathbf{q})_{\mathbf{q} \in \mathbb{Z}^n \setminus \{\mathbf{0}\}}$ be a sequence of vectors in $\mathbb{R}^m$. We call $\mathbf{y}_\mathbf{q}$ the \emph{moving targets}, and when $\mathbf{y}_\mathbf{q} = \mathbf{y}$ is fixed for every $\mathbf{q}$, then we call $\mathbf{y}$ the \emph{inhomogeneous parameter}. For each $\mathbf{q} \in \mathbb{Z}^n \setminus \{\mathbf{0}\}$ and $\delta \in \mathbb{R}_{\geq 0},$ we define the following sets:
\begin{align*}
    A_{n,m}(\mathbf{q},\delta) &= \{\mathbf{x} \in [0,1)^{nm} : \exists \mathbf{p} \in \mathbb{Z}^m \text{ s.t. } |\mathbf{qx} - \mathbf{p} - \mathbf{y}_\mathbf{q}| < \delta\}, \\
    \tilde{A}_{n,m}(\mathbf{q},\delta) &= \{\mathbf{x} \in [0,1)^{nm} : \exists \mathbf{p} \in \mathbb{Z}^m \text{ s.t. } |\mathbf{qx} - \mathbf{p} - \mathbf{y}_\mathbf{q}| <\delta, \ \gcd(\mathbf{q}, b_\mathbf{q}\mathbf{p} + \mathbf{a}_\mathbf{q}) = 1\}.
\end{align*}
    When a \emph{multivariate function} $\Psi : \mathbb{Z}^n \setminus \{\mathbf{0}\} \to \mathbb{R}_{\geq 0}$ is given, we denote $A_\mathbf{q} := A_{n,m}(\mathbf{q}, \Psi(\mathbf{q})),\ \tilde{A}_\mathbf{q} := \tilde{A}_{n,m}(\mathbf{q}, \Psi(\mathbf{q}))$. Similarly, when a \emph{univariate function} $\psi : \mathbb{N} \to \mathbb{R}_{\geq 0}$ is given, we denote $A_\mathbf{q} := A_{n,m}(\mathbf{q}, \psi(|\mathbf{q}|)),\ \tilde{A}_\mathbf{q} := \tilde{A}_{n,m}(\mathbf{q}, \psi(|\mathbf{q}|))$.
\end{definition}

Throughout this paper, $\ll$ and $\asymp$ denote the Vinogradov notations. The implicit constants in $\ll$ and $\asymp$ may depend on the dimension $(n,m)$ only. In proving Theorems \ref{thm1} and \ref{thm2}, we may assume that $\psi(q)<\frac{1}{2}$ for all $q\in \mathbb{N}$, since if we denote $\psi'(q) = \min\{\psi(q), c\}$ for some $0<c<\frac{1}{2},$ then $\sum_{q=1}^\infty q^{n-1}\psi(q)^m = \infty$ is equivalent to $\sum_{q=1}^\infty q^{n-1}\psi'(q)^m = \infty$ and $\mathscr{A}_{n,m}^\mathbf{y} (\psi') \subseteq \mathscr{A}_{n,m}^\mathbf{y} (\psi).$ Therefore, we assume $\delta< \frac{1}{2}$ and $\psi: \mathbb{N} \to [0,\frac{1}{2})$ in the following lemmas.

The following lemma gives the measure of the set $\tilde{A}_{n,m}(\mathbf{q}, \delta)$, and can be seen as a generalization of \cite[Lemma 7]{BV09}, which concerns only the case $\mathbf{a}_\mathbf{q} = \mathbf{0}$ and $b_\mathbf{q}=1$, to the setting of arbitrary $(\mathbf{a}_\mathbf{q}, b_\mathbf{q})$ with $\gcd(\mathbf{a}_\mathbf{q}, b_\mathbf{q}) = 1$. Alternatively, it can also be seen as a generalization of \cite[Proposition 1]{BHV24} to arbitrary dimensions.

\begin{lemma}\label{measure}
    Let $n,m \in \mathbb{N}$, $\delta \in [0,\frac{1}{2}),$ and let $\frac{\mathbf{a}_\mathbf{q}}{b_\mathbf{q}}, \mathbf{y}_\mathbf{q}$ be defined as in Definition \ref{def}. Then we have
    $$|\tilde{A}_{n,m}(\mathbf{q},\delta)| = (2\delta)^m \prod_{p|\mathbf{q}, p \nmid b_{\mathbf{q}}}(1-p^{-m}),$$
    where $p$ runs over primes, and $p|\mathbf{q}$ means that $p|\gcd(\mathbf{q})$.
\end{lemma}
\begin{proof}
    We follow the strategy of \cite[Lemmas 6 and 7]{BV09}. Let 
    $$A(\mathbf{q}, \mathbf{p}, \delta) := A^{\mathbf{y}_\mathbf{q}} (\mathbf{q}, \mathbf{p}, \delta) := \{\mathbf{x} \in [0,1)^{nm} : |\mathbf{qx} - \mathbf{p} - \mathbf{y}_\mathbf{q}| < \delta \}.$$
    Then for any integer $l | \mathbf{q}$ and a vector $\mathbf{j} \in \mathbb{Z}^m$, we have
    $$\bigcup_{\mathbf{k} \in \mathbb{Z}^m} A^{\mathbf{y}_\mathbf{q}} (\mathbf{q}, l\mathbf{k} + \mathbf{j}, \delta) = A^{(\mathbf{y}_\mathbf{q} + \mathbf{j})/l}(\mathbf{q}/l, \delta/l),$$
    where $A^{(\mathbf{y}_\mathbf{q} + \mathbf{j})/l}(\mathbf{q}/l, \delta/l)= \{ \mathbf{x} \in [0,1)^{nm} : \exists \mathbf{p} \in \mathbb{Z}^m \text{ s.t. } |\frac{\mathbf{q}}{l}\mathbf{x} - \mathbf{p} - \frac{\mathbf{y}_\mathbf{q} + \mathbf{j}}{l}| <\frac{\delta}{l}\}$. Since $\delta<\frac{1}{2}$, each $A^{\mathbf{y}_\mathbf{q}} (\mathbf{q}, l\mathbf{k} + \mathbf{j}, \delta)$ is disjoint from the others, and thus
    $$\sum_{\mathbf{k} \in \mathbb{Z}^m} |A^{\mathbf{y}_\mathbf{q}} (\mathbf{q}, l\mathbf{k} + \mathbf{j}, \delta)| = \left|A^{(\mathbf{y}_\mathbf{q} + \mathbf{j})/l}(\mathbf{q}/l, \delta/l)\right| = \left(\frac{2\delta}{l}\right)^m.$$
    Recall the well known properties of the M\"obius function $\mu$ :
    \begin{equation}\label{mobius}
        \sum_{l|d} \mu(l)= \begin{cases} 0 \ \text{ if } \ d>1 \\ 1 \ \text{ if } \ d=1 \end{cases} \quad \text{ and } \quad \sum_{l|d, \gcd(l,b)=1} \frac{\mu(l)}{l^m} = \prod_{p|d, p\nmid b} (1-p^{-m}).
    \end{equation}
    Using the notation $A\times 0 := \emptyset,\ A\times 1 := A$, we have
    $$\tilde{A}_{n,m}(\mathbf{q},\delta) = \bigcup_{\mathbf{p} \in \mathbb{Z}^m} A(\mathbf{q}, \mathbf{p}, \delta) \sum_{l|\gcd(\mathbf{q}, b_\mathbf{q}\mathbf{p} + \mathbf{a}_\mathbf{q})} \mu(l).$$
    If $\gcd(l,b_\mathbf{q}) \neq 1$, then $l\nmid b_\mathbf{q}\mathbf{p} + \mathbf{a}_\mathbf{q}$ for any $\mathbf{p} \in \mathbb{Z}^m$ since $\gcd(b_\mathbf{q}, \mathbf{a}_\mathbf{q}) = 1$. When $\gcd(l, b_\mathbf{q})= 1$, we can find an inverse $b_\mathbf{q}^{(l)} \in \mathbb{N}$ of $b_\mathbf{q}$ modulo $l$ satisfying $b_\mathbf{q}^{(l)} b_\mathbf{q} \equiv 1 \pmod{l},$ and so $b_\mathbf{q} \mathbf{p} + \mathbf{a}_\mathbf{q} \equiv \mathbf{0} \pmod{l} \Leftrightarrow \mathbf{p} \equiv -b^{(l)}_\mathbf{q} \mathbf{a}_\mathbf{q} \pmod{l}$, thus we have
    $$\bigcup_{\substack{\mathbf{p} \in \mathbb{Z}^m \\ l|b_\mathbf{q} \mathbf{p} + \mathbf{a}_\mathbf{q}}} A(\mathbf{q}, \mathbf{p}, \delta) = \bigcup_{\mathbf{k} \in \mathbb{Z}^m} A(\mathbf{q}, l\mathbf{k} - b^{(l)}_\mathbf{q} \mathbf{a}_\mathbf{q}, \delta).$$
    Therefore we have
    \begin{align*}
        |\tilde{A}_{n,m}(\mathbf{q},\delta)| &= \sum_{\mathbf{p} \in \mathbb{Z}^m} |A(\mathbf{q}, \mathbf{p}, \delta)| \sum_{l|\gcd(\mathbf{q}, b_\mathbf{q}\mathbf{p} + \mathbf{a}_\mathbf{q})} \mu(l) \\
        &= \sum_{l|\mathbf{q}} \sum_{\mathbf{p} \in \mathbb{Z}^m} |A(\mathbf{q}, \mathbf{p}, \delta)| \sum_{l | b_\mathbf{q}\mathbf{p} + \mathbf{a}_\mathbf{q}} \mu(l) \\
        &= \sum_{l|\mathbf{q}, \gcd(l, b_\mathbf{q}) = 1}\mu(l) \sum_{\mathbf{k} \in \mathbb{Z}^m} |A(\mathbf{q}, l\mathbf{k} - b^{(l)}_\mathbf{q} \mathbf{a}_\mathbf{q}, \delta)| \\
        &= \sum_{l|\mathbf{q}, \gcd(l, b_\mathbf{q}) = 1}\mu(l) \left(\frac{2\delta}{l}\right)^m = (2\delta)^m \prod_{p|\mathbf{q}, p\nmid b_\mathbf{q}}(1-p^{-m}).
    \end{align*}
\end{proof}

\begin{lemma}\label{comparable}
    Let $nm \geq 2$, $\psi : \mathbb{N} \to [0,\frac{1}{2}),$ and let $\frac{\mathbf{a}_\mathbf{q}}{b_\mathbf{q}}, \mathbf{y}_\mathbf{q}$ be as in Definition \ref{def}. Then
    $$\sum_{|\mathbf{q}| =1}^Q \left|\tilde{A}_{n,m}(\mathbf{q}, \psi(|\mathbf{q}|))\right| \asymp \sum_{q=1}^Q q^{n-1}\psi(q)^m.$$
\end{lemma}
\begin{proof}
    Recall that $A_\mathbf{q} = A_{n,m}(\mathbf{q}, \psi(|\mathbf{q}|)),\ \tilde{A}_\mathbf{q} = \tilde{A}_{n,m}(\mathbf{q}, \psi(|\mathbf{q}|))$ and define
    $$A'_\mathbf{q} := \{\mathbf{x} \in [0,1)^{nm} : \exists \mathbf{p} \in \mathbb{Z}^m \text{ s.t. } |\mathbf{qx} - \mathbf{p}| < \psi(|\mathbf{q}|), \ \gcd(\mathbf{q},\mathbf{p}) = 1\}.$$ The result follows from $\sum_{|\mathbf{q}| =1}^Q |A'_\mathbf{q}| \asymp \sum_{q=1}^Q q^{n-1}\psi(q)^m$ \cite[Lemma 10]{BV09}, and the inequality $|A'_\mathbf{q}| \leq |\tilde{A}_\mathbf{q}| \leq |A_\mathbf{q}| = (2\psi(|\mathbf{q}|))^m$ from
    $$|A'_\mathbf{q}| = (2\psi(|\mathbf{q}|))^m \prod_{p|\mathbf{q}} (1-p^{-m}), \qquad |\tilde{A}_\mathbf{q}| = (2\psi(|\mathbf{q}|))^m\prod_{p|\mathbf{q}, p \nmid b_{\mathbf{q}}}(1-p^{-m}).$$
\end{proof}

\begin{lemma}[\cite{Spr79}, Chap 1, Lemma 9] \label{notparallel}
    Let $n \geq 2,\ m \geq 1,\ \delta_1, \delta_2 \in [0,\frac{1}{2}),$ and let $\frac{\mathbf{a}_\mathbf{q}}{b_\mathbf{q}}, \mathbf{y}_\mathbf{q}$ be as in Definition \ref{def}. If $\mathbf{q} \nparallel \mathbf{r}$, then $|A_{n,m}(\mathbf{q}, \delta_1) \cap A_{n,m}(\mathbf{r}, \delta_2)| = |A_{n,m}(\mathbf{q},\delta_1)| |A_{n,m}(\mathbf{r},\delta_2)|$. In particular, if $\mathbf{q} \nparallel \mathbf{r}$, then $|\tilde{A}_{n,m}(\mathbf{q}, \delta_1) \cap \tilde{A}_{n,m}(\mathbf{r}, \delta_2)| \leq |A_{n,m}(\mathbf{q},\delta_1)| |A_{n,m}(\mathbf{r}, \delta_2)|.$
\end{lemma}

\begin{lemma}\label{onedimension}
    Let $n,m \in \mathbb{N}$, and let $\mathbf{y} \in \mathbb{R}^m$ be an inhomogeneous parameter. Suppose that $(\mathbf{a}_\mathbf{q}, b_\mathbf{q})$ depends only on $d= \gcd(\mathbf{q})$, that is, $(\mathbf{a}_\mathbf{q}, b_\mathbf{q}) = (\mathbf{a}_d, b_d)$. Then we have
    $$|\tilde{A}_{n,m}(\mathbf{q}, \delta)| = |\tilde{A}_{1,m}(d, \delta)|.$$
\end{lemma}
\begin{proof}
    There exists a primitive vector $\mathbf{k} \in \mathbb{Z}^n \setminus \{\mathbf{0}\}$, that is, a vector satisfying $\gcd(\mathbf{k}) = 1$, such that $\mathbf{q} = d\mathbf{k}.$ Consider the transformation
    $$T_\mathbf{k} : [0,1)^{nm} \to [0,1)^m ;\ \mathbf{x} \mapsto \mathbf{kx} \text{ mod } 1.$$
    We have the equivalences
    \begin{align*}
        \mathbf{x} \in \tilde{A}_{n,m}(\mathbf{q}, \delta) 
        &\Leftrightarrow \exists \mathbf{p} \in \mathbb{Z}^m \text{ s.t. }|d\mathbf{kx} - \mathbf{p} - \mathbf{y}| < \delta  \\
        &\qquad \qquad \text{ with } \gcd(d\mathbf{k}, b_\mathbf{q}\mathbf{p} + \mathbf{a}_\mathbf{q}) = \gcd(d, b_d \mathbf{p} + \mathbf{a}_d) = 1 \\
        &\Leftrightarrow T_\mathbf{k}(\mathbf{x}) \in \tilde{A}_{1,m}(d, \delta),
    \end{align*}
    hence $\tilde{A}_{n,m}(\mathbf{q}, \delta) = T_\mathbf{k}^{-1}(\tilde A_{1,m}(d, \delta)).$ The lemma follows from the fact that $T_\mathbf{k}$ is measure-preserving \cite[Chap 1, Lemma 8]{Spr79}.
\end{proof}

\begin{lemma}\label{intersection}
    Under the same assumption as in Lemma \ref{onedimension}, let $\mathbf{k} \in \mathbb{Z}^n \setminus \{\mathbf{0}\}$ be a primitive vector and $\delta_1, \delta_2 \in [0,\frac{1}{2}).$ Set $\mathbf{q} = d\mathbf{k}$ and $\mathbf{r}=e \mathbf{k}$ for $d, e \in \mathbb{Z}\setminus \{0\}.$ Then
    $$|\tilde A_{n,m}(\mathbf q,\delta_1) \cap \tilde A_{n,m}(\mathbf r, \delta_2)| = |\tilde A_{1,m}(d, \delta_1) \cap \tilde A_{1,m}(e, \delta_2)|.$$
\end{lemma}
\begin{proof}
    Since $\gcd(\mathbf{q}) = |d|$ and $\gcd(\mathbf{r}) = |e|$, we have the equivalences
    \begin{align*}
        \mathbf{x} \in \tilde{A}_{n,m}&(\mathbf{q}, \delta_1) \cap \tilde{A}_{n,m}(\mathbf{r}, \delta_2) \\
        &\Leftrightarrow \exists \mathbf{p}, \mathbf{s}\in \mathbb{Z}^m \text{ s.t. } \begin{cases} |d\mathbf{kx} - \mathbf{p} - \mathbf{y}| < \delta_1 \text{ with } \gcd(d\mathbf{k}, b_\mathbf{q}\mathbf{p} + \mathbf{a}_\mathbf{q}) = 1, \\ |e\mathbf{kx} - \mathbf{s} - \mathbf{y}| < \delta_2 \text{ with } \gcd(e\mathbf{k}, b_\mathbf{r}\mathbf{s} + \mathbf{a}_\mathbf{r}) = 1 \end{cases} \\
        &\Leftrightarrow \exists \mathbf{p}, \mathbf{s}\in \mathbb{Z}^m \text{ s.t. } \begin{cases} |d\mathbf{kx} - \mathbf{p} - \mathbf{y}| < \delta_1 \text{ with } \gcd(d, b_{|d|} \mathbf{p} + \mathbf{a}_{|d|}) = 1, \\ |e\mathbf{kx} - \mathbf{s} - \mathbf{y}| < \delta_2 \text{ with } \gcd(e, b_{|e|} \mathbf{s} + \mathbf{a}_{|e|}) = 1 \end{cases} \\
        &\Leftrightarrow T_\mathbf{k}(\mathbf{x}) \in \tilde A_{1,m}(d, \delta_1) \cap \tilde A_{1,m}(e, \delta_2),
    \end{align*}
    hence $\tilde A_{n,m} (\mathbf q, \delta_1) \cap \tilde A_{n,m} (\mathbf r, \delta_2) = T_\mathbf{k}^{-1}(\tilde A_{1,m} (d,\delta_1) \cap \tilde A_{1,m}(e, \delta_2)).$ The lemma follows from the fact that $T_\mathbf{k}$ is measure-preserving.
\end{proof}

\begin{lemma}\label{basicineq}
    Let $(n,m)= (1,1),\ \delta_1, \delta_2 \in [0,\frac{1}{2}),\ d,e \in \mathbb{Z} \setminus \{0\}$, and let $\frac{a_q}{b_q}, y_q$ be as in Definition \ref{def}. Then we have
    $$|\tilde A_{1,1}(d, \delta_1) \cap \tilde A_{1,1}(e, \delta_2)| \ll \delta_1 \delta_2 + \frac{\delta_1 \gcd(d,e)}{|d|}.$$
\end{lemma}
\begin{proof}
    It follows from $|A_{1,1}(d, \delta_1) \cap A_{1,1}(e, \delta_2)| \ll \delta_1 \delta_2 + \frac{\delta_1 \gcd(d,e)}{|d|}$ \cite[Lemma 8]{AR22} and inequality $|\tilde A_{1,1}(d, \delta_1) \cap \tilde A_{1,1}(e, \delta_2)| \leq |A_{1,1}(d, \delta_1) \cap A_{1,1}(e, \delta_2)|.$
\end{proof}

\vspace{.1cm}
\section{Proof of Theorem \ref{thm1}}
\subsection{Quasi-independence on average}
Let $(n,m) = (2,1)$, $y\in \mathbb{R}$ be an inhomogeneous parameter, and $\psi:\mathbb{N} \to \mathbb{R}_{\geq 0}$ satisfy $\sum_{q=1}^\infty q\psi(q) = \infty.$ We may assume that $\psi(q) \leq q^{-1}$ for all $q \in \mathbb{N}$, since if we denote $\psi'(q) = \min\{\psi(q), q^{-1}\},$ then $\sum_{q=1}^\infty q\psi(q) = \infty$ is equivalent to $\sum_{q=1}^\infty q\psi'(q) = \infty$ and $\mathscr{A}^y_{2,1}(\psi') \subseteq \mathscr{A}^y_{2,1}(\psi)$.

For the inhomogeneous parameter $y\in \mathbb{R}$, we define the sequences $(a_d, b_d)_{d \in \mathbb{Z} \setminus\{0\}}$ and $(a_\mathbf{q}, b_\mathbf{q})_{\mathbf{q} \in \mathbb{Z}^2 \setminus \{\mathbf{0}\}}$ as given in (\ref{a,b}) and (\ref{a,b,q}). We now present the main lemma of this article.

\begin{lemma}\label{mainlemma}
    Let $\psi: \mathbb{N} \to [0,\frac{1}{2})$ satisfy $\psi(q) \leq q^{-1}$ for all $q \in \mathbb{N},$ and let $y \in \mathbb{R}$ denote an inhomogeneous parameter. Define $(a_d,b_d)_{d \in \mathbb{Z} \setminus \{0\}}$ as in (\ref{a,b}), and  $\tilde{A}$ as in Definition \ref{def}. Let integers $q,r,d,e$ satisfy $1\leq r < q$ and $1\leq |e| < |d|$. If $\gcd(d,e) \geq 3$ and $r \geq 2d^2,$ then
    $$|\tilde A_{1,1} (d, \psi(q)) \cap \tilde A_{1,1}(e, \psi(r))| = 0.$$
\end{lemma}
\begin{proof}
    We are going to show that $\tilde A_{1,1} (d, \psi(q)) \cap A_{1,1}(e, \psi(r)) = \emptyset$ which is a stronger result. The definitions of the sets $\tilde A_{1,1}(d, \psi(q))$ and $A_{1,1}(e, \psi(r))$ give
    \begin{equation}\label{union}
        \begin{aligned}
            \tilde A_{1,1}(d, \psi(q)) &= \bigcup_{\substack{p \in \mathbb{Z}\\ \gcd(d, b_d p + a_d)=1}} \left(\frac{p+y}{d} -\frac{\psi(q)}{|d|}, \frac{p+y}{d} + \frac{\psi(q)}{|d|} \right) \cap [0,1),\\
            A_{1,1}(e, \psi(r)) &= \bigcup_{s \in \mathbb{Z}} \left(\frac{s+y}{e} - \frac{\psi(r)}{|e|}, \frac{s+y}{e} + \frac{\psi(r)}{|e|} \right) \cap [0,1).
        \end{aligned}
    \end{equation}
    
    The distance between the centers of intervals in the unions (\ref{union}) is
    \begin{align}\label{distance}
        \left| \frac{p+y}{d} - \frac{s+y}{e} \right| &= \left| \frac{p+ \frac{a_d}{b_d} +y -\frac{a_d}{b_d}}{d} - \frac{s+ \frac{a_d}{b_d} + y - \frac{a_d}{b_d}}{e} \right| \nonumber \\
        &= \left| \frac{1}{b_d} \left(\frac{b_d p + a_d}{d} - \frac{b_d s + a_d}{e} \right) + \left(y - \frac{a_d}{b_d} \right) \left(\frac{1}{d} - \frac{1}{e} \right)\right|.
    \end{align}
    Since $\gcd(d, b_d p+a_d) = 1$ and $|e| < |d|$, $\frac{b_d p+a_d}{d}$ is a reduced fraction with a denominator larger than that of $\frac{b_d s+a_d}{e}.$ Thus the two fractions $\frac{b_d p+a_d}{d}$ and $\frac{b_d s+a_d}{e}$ are distinct, so the first term of (\ref{distance}) satisfies
    $$\left| \frac{1}{b_d} \left(\frac{b_d p + a_d}{d} - \frac{b_d s + a_d}{e} \right) \right| \geq \frac{1}{b_d \text{lcm}(d,e)} \geq \frac{3}{b_d |de|},$$
    where the last inequality comes from $\gcd(d,e) \geq 3$. By the definition of $(a_d,b_d)$, we have the inequality $|b_d y - a_d| < \frac{1}{|d|}$. Since $|e|<|d|$, we also have $|e-d|<2|d|$, thus the second term of (\ref{distance}) satisfies
    $$\left|\left(y - \frac{a_d}{b_d} \right) \left(\frac{1}{d} - \frac{1}{e} \right)\right| < \frac{1}{b_d|d|} \frac{|e-d|}{|de|} < \frac{2}{b_d |de|}.$$
    Therefore, from (\ref{distance}) we have
    $$\left| \frac{p+y}{d} - \frac{s+y}{e} \right| > \frac{1}{b_d |de|}.$$

    We now compare the distance between centers with the sum of the radii of intervals in the unions (\ref{union}), given by $\frac{\psi(q)}{|d|} + \frac{\psi(r)}{|e|}$. Since $\psi(q) \leq q^{-1}$, $2d^2 \leq r<q,$ and $|e|<|d|$, it follows that
    $$\frac{\psi(q)}{|d|} + \frac{\psi(r)}{|e|} \leq \frac{1}{q|d|} + \frac{1}{r|e|} < \frac{2}{r|e|} \leq \frac{1}{d^2|e|}.$$
    By the definition of $b_d$, we have $b_d \leq |d|$. Therefore
    $$\frac{\psi(q)}{|d|} + \frac{\psi(r)}{|e|} < \frac{1}{b_d |de|} < \left| \frac{p+y}{d} - \frac{s+y}{e} \right|$$
    for any $p,s \in \mathbb{Z}$ such that $\gcd(d, b_dp+a_d) = 1$. This implies that
    $$\left(\frac{p+y}{d} -\frac{\psi(q)}{|d|}, \frac{p+y}{d} + \frac{\psi(q)}{|d|} \right) \cap \left(\frac{s+y}{e} - \frac{\psi(r)}{|e|}, \frac{s+y}{e} + \frac{\psi(r)}{|e|} \right) = \emptyset,$$
    which allows us to conclude that $\tilde A_{1,1} (d, \psi(q)) \cap A_{1,1}(e, \psi(r)) = \emptyset$.
\end{proof}
The following theorem shows that the sequence $(\tilde{A}_\mathbf{q})_{\mathbf{q} \in \mathbb{Z}^2 \setminus \{\mathbf{0}\}}$ is QIA.
\begin{theorem}\label{thm3}
    Let $(n,m) = (2,1)$ and let $y \in \mathbb{R}$ be an inhomogeneous parameter. Suppose $\psi : \mathbb{N} \to [0,\frac{1}{2})$ satisfies $\psi(q) \leq q^{-1}$ for all $q \in \mathbb{N}$ and $\sum_{q=1}^\infty q\psi(q) =\infty$. Define $(a_\mathbf{q}, b_\mathbf{q})_{\mathbf{q} \in \mathbb{Z}^2 \setminus \{\mathbf{0}\}}$ as in (\ref{a,b,q}) and  $\tilde{A}_\mathbf{q}=\tilde{A}_{2,1}(\mathbf{q}, \psi(|\mathbf{q}|))$ as in Definition \ref{def}. Then $(\tilde A_\mathbf{q})_{\mathbf{q} \in \mathbb{Z}^2 \setminus \{\mathbf{0}\}}$ is quasi-independent on average.
\end{theorem}
\begin{proof}
    Our goal is to show that $\sum_{1 \leq |\mathbf{q}|,|\mathbf{r}| \leq Q} |\tilde A_\mathbf{q} \cap \tilde A_\mathbf{r}| \ll \left( \sum_{1\leq |\mathbf{q}| \leq Q} |\tilde A_\mathbf{q}| \right)^2.$ We first reduce the problem to the one-dimensional case in Step 1, and then bound the arising terms in Step 2 and 3.
    
    \textit{Step 1. Reduction to 1-dimension:} By Lemma \ref{comparable} and Lemma \ref{notparallel}, it is enough to show that
    $$\sum_{\substack{1\leq |\mathbf{q}|,|\mathbf{r}| \leq Q \\ \mathbf{q} \parallel \mathbf{r}}} |\tilde A_\mathbf{q} \cap \tilde A_\mathbf{r}| \ll \left( \sum_{1\leq |\mathbf{q}| \leq Q} |A_\mathbf{q}| \right)^2.$$
    Since $\sum_{|\mathbf{q}|=1}^\infty |A_\mathbf{q}| =\infty$, considering large enough $Q$ so that $\sum_{|\mathbf{q}|=1}^Q |A_\mathbf{q}| \geq 2$,
    \begin{align*}
        \sum_{1\leq |\mathbf{q}| \leq Q} \sum_{\substack{|\mathbf{r}| = |\mathbf{q}| \\ \mathbf{r} \parallel \mathbf{q}}} |\tilde A_\mathbf{q} \cap \tilde A_\mathbf{r}| &= \sum_{1\leq |\mathbf{q}| \leq Q} \left( |\tilde A_\mathbf{q} \cap \tilde A_\mathbf{q}| + |\tilde A_\mathbf{q} \cap \tilde A_\mathbf{-q}| \right) \\
        &\leq 2 \sum_{1\leq |\mathbf{q}| \leq Q} |A_\mathbf{q}| \leq \left( \sum_{1\leq |\mathbf{q}| \leq Q} |A_\mathbf{q}| \right)^2.
    \end{align*}
    Therefore, it is enough to show that
    \begin{equation}\label{QIA}
        \sum_{1\leq q \leq Q} \sum_{|\mathbf{q}|=q} \sum_{1\leq r < q} \sum_{\substack{|\mathbf{r}| = r \\ \mathbf{r} \parallel \mathbf{q}}} |\tilde A_\mathbf{q} \cap \tilde A_\mathbf{r}| \ll \left( \sum_{1\leq |\mathbf{q}| \leq Q} |A_\mathbf{q}| \right)^2.
    \end{equation}
    For a fixed $1 \leq q\leq Q$, the sum in the left hand side of (\ref{QIA}) becomes
    \begin{equation}\label{qfixed}
        \sum_{|\mathbf{q}|=q} \sum_{1\leq r<q} \sum_{\substack{|\mathbf{r}| = r \\ \mathbf{r} \parallel \mathbf{q}}} |\tilde A_\mathbf{q} \cap \tilde A_\mathbf{r}| = \sum_{\substack{d \in \mathbb{N} \\ d \mid q}} \sum_{\substack{|\mathbf{q}|= q \\ \gcd(\mathbf{q}) = d}} \sum_{1\leq r < q} \sum_{\substack{|\mathbf{r}| = r \\ \mathbf{r} \parallel \mathbf{q}}} |\tilde A_\mathbf{q} \cap \tilde A_\mathbf{r}|.
    \end{equation}
    For a fixed $\mathbf{q} \in \mathbb{Z}^2$ with $|\mathbf{q}| = q$ and $\gcd(\mathbf{q}) = d$, we have a primitive vector $\mathbf{k} \in \mathbb{Z}^2$ such that $\mathbf{q} = d\mathbf{k}$. If $\mathbf{r} \parallel \mathbf{q}$ with $|\mathbf{r}| =r$, then $\mathbf{r} \parallel \mathbf{k}$. Since $|\mathbf{k}| = \frac{q}{d}$, it follows that $\frac{q}{d} \mid r$ and $\mathbf{r} = \pm \frac{r}{q/d} \mathbf{k}$, where the sign is determined by whether $\mathbf{q}$ and $\mathbf{r}$ point in the same or opposite direction. Let $e = \pm \frac{r}{q/d} \in \mathbb{Z} \setminus \{0\}$ so that $\mathbf{r} = e \mathbf{k}$ and $1 \leq |e|<d$. By Lemma \ref{intersection},
    \begin{align*}
        \sum_{1\leq r<q} \sum_{\substack{|\mathbf{r}| = r \\ \mathbf{r} \parallel \mathbf{q}}} |\tilde A_\mathbf q \cap \tilde A_\mathbf r| &= \sum_{\substack{1 \leq r <q \\ \frac{q}{d} \mid r \\ (e = \pm rd/q)}} |\tilde A_{d\mathbf{k}} \cap \tilde A_{e\mathbf{k}}|\\
        &= \sum_{\substack{1 \leq |e| < d \\ (r=|qe/d|)}} |\tilde A_{2,1}(d\mathbf{k}, \psi(q)) \cap \tilde A_{2,1}(e\mathbf{k}, \psi(r))| \\
        &= \sum_{\substack{1 \leq |e| < d \\ (r=|qe/d|)}} |\tilde A_{1,1}(d, \psi(q)) \cap \tilde A_{1,1}(e, \psi(r))|
    \end{align*}
    Therefore, (\ref{qfixed}) equals to
    \begin{equation}\label{twosums}
        \begin{aligned}
            \sum_{\substack{d \in \mathbb{N} \\ d \mid q}} &\sum_{\substack{|\mathbf{q}|= q \\ \gcd(\mathbf{q}) = d}} \sum_{\substack{1 \leq |e| < d \\ \gcd(d,e) \geq 3 \\ (r=|qe/d|)}} |\tilde A_{1,1}(d, \psi(q)) \cap \tilde A_{1,1}(e, \psi(r))| \\
            &+ \sum_{\substack{d \in \mathbb{N} \\ d \mid q}} \sum_{\substack{|\mathbf{q}|= q \\ \gcd(\mathbf{q}) = d}} \sum_{\substack{1 \leq |e| < d \\ \gcd(d,e) \leq 2 \\ (r=|qe/d|)}} |\tilde A_{1,1}(d, \psi(q)) \cap \tilde A_{1,1}(e, \psi(r))|
        \end{aligned}
    \end{equation}

    \textit{Step 2. First term of (\ref{twosums}):} This argument follows from \cite[Lemma 6]{Hau23}, but we include it here for completeness. By Lemma \ref{mainlemma}, we only need to consider the case $r<2d^2$. Since $r<2d^2$ is equivalent to  $|e| = \frac{rd}{q} < \frac{2d^3}{q},$ we have
    \begin{align*}
        \sum_{\substack{d \in \mathbb{N} \\ d \mid q}} \sum_{\substack{|\mathbf{q}|= q \\ \gcd(\mathbf{q}) = d}} &\sum_{\substack{1 \leq |e| < d \\ \gcd(d,e) \geq 3 \\ (r=|qe/d|)}} |\tilde A_{1,1}(d, \psi(q)) \cap \tilde A_{1,1}(e, \psi(r))| \\
        &\leq \sum_{\substack{d \in \mathbb{N} \\ d \mid q}} \sum_{\substack{|\mathbf{q}|= q \\ \gcd(\mathbf{q}) = d}} \sum_{\substack{1 \leq |e| < d \\ |e| < 2d^3/q \\ (r=|qe/d|)}} |\tilde A_{1,1}(d, \psi(q)) \cap \tilde A_{1,1}(e, \psi(r))| \\
        &\overset{\text{Lemma \ref{basicineq}}}{\ll} \sum_{\substack{d \in \mathbb{N} \\ d \mid q}} \sum_{\substack{|\mathbf{q}|= q \\ \gcd(\mathbf{q}) = d}} \sum_{\substack{1 \leq |e| < d \\ |e| < 2d^3/q \\ (r=|qe/d|)}} \left(\psi(q)\psi(r) + \frac{\psi(q)\gcd(d,e)}{d} \right).
    \end{align*}
    The first term $\psi(q)\psi(r)$ in the last sum equals to $\frac{1}{4}|A_\mathbf{q}||A_\mathbf{r}|,$ so it satisfies the QIA condition when summed over $1 \leq q \leq Q$. Therefore, we only need to consider the second term $\frac{\psi(q)\gcd(d,e)}{d}$. Since the number of vectors $\mathbf{q}$ satisfying $|\mathbf{q}| = q$ and $\gcd(\mathbf{q}) = d$ is comparable to $\varphi(\frac{q}{d}) \leq \frac{q}{d},$ where $\varphi(\cdot)$ denotes the Euler totient function, it follows that
    \begin{align*}
        \sum_{\substack{d \in \mathbb{N} \\ d \mid q}} &\sum_{\substack{|\mathbf{q}|= q \\ \gcd(\mathbf{q}) = d}} \sum_{\substack{1 \leq |e| < d \\ |e| < 2d^3/q}} \frac{\psi(q)\gcd(d,e)}{d}
        \ll \sum_{\substack{d \in \mathbb{N} \\ d \mid q}} \frac{q}{d} \sum_{\substack{1 \leq |e| < d \\ |e| < 2d^3/q}} \frac{\psi(q)\gcd(d,e)}{d} \\
        &= q\psi(q) \sum_{\substack{d \in \mathbb{N} \\ d \mid q}} \sum_{\substack{1 \leq |e| < d \\ |e| < 2d^3/q}} \frac{\gcd(d,e)}{d^2}
        \leq q\psi(q) \sum_{\substack{d \mid q \\ d > (q/2)^{1/3}}} \frac{1}{d^2} \sum_{1\leq |e| <d} \gcd(d,e)\\
        &\ll q\psi(q) \sum_{\substack{d \mid q \\ d > (q/2)^{1/3}}} \frac{\tau(d)}{d} \leq q\psi(q) \tau(q) \frac{\tau(q)}{(q/2)^{1/3}} \asymp q^{2/3} \psi(q) \tau(q)^2 \ll q \psi(q).
    \end{align*}
    Here $\tau(\cdot)$ denotes the number of divisors, thus $\sum_{1\leq e\leq d} \gcd(d,e) \leq d\tau(d),$ and the last inequality comes from the fact that $\tau(q) \ll q^{\epsilon}$ for any $\epsilon>0$ \cite[Thm 315]{HW08}. Therefore the first term of (\ref{twosums}) meets the QIA condition.

    \textit{Step 3. Second term of (\ref{twosums}):} By the same logic, we only need to consider
    \begin{align*}
        \sum_{\substack{d \in \mathbb{N} \\ d \mid q}} &\sum_{\substack{|\mathbf{q}|= q \\ \gcd(\mathbf{q}) = d}} \sum_{\substack{1 \leq |e| < d \\ \gcd(d,e) \leq 2}} \frac{\psi(q)\gcd(d,e)}{d} \\
        &= \sum_{\substack{d \in \mathbb{N} \\ d \mid q}} \sum_{\substack{|\mathbf{q}|= q \\ \gcd(\mathbf{q}) = d}} \sum_{\substack{1 \leq |e| < d \\ \gcd(d,e) =1}} \frac{\psi(q)}{d} + \sum_{\substack{d \in \mathbb{N} \\ d \mid q}} \sum_{\substack{|\mathbf{q}|= q \\ \gcd(\mathbf{q}) = d}} \sum_{\substack{1 \leq |e| < d \\ \gcd(d,e) =2}} \frac{2\psi(q)}{d} \\
        &\ll \psi(q) \sum_{\substack{d \in \mathbb{N} \\ d \mid q}} \varphi \left(\frac{q}{d}\right) \left( \sum_{\substack{1 \leq |e| < d \\ \gcd(d,e) =1}} \frac{1}{d} + \sum_{\substack{1 \leq |e| < d \\ \gcd(d,e) =2}} \frac{2}{d} \right) \\
        &\ll \psi(q) \sum_{\substack{d \in \mathbb{N} \\ d \mid q}} \frac{\varphi(q/d) \varphi(d)}{d}  \leq \psi(q) \varphi(q) \sum_{\substack{d \in \mathbb{N} \\ d \mid q}} \frac{1}{d} = \psi(q)\frac{\varphi(q) \sigma(q)}{q} <q\psi(q),
    \end{align*}
    where $\sigma(q) = \sum_{d|q} d$ and the last inequality comes from \cite[Thm 329]{HW08}.
\end{proof}

\vspace{.1cm}
\subsection{QIA gives full measure}
We now prove that showing QIA is sufficient to guarantee full measure. While our main interest lies in the case of an inhomogeneous parameter and a univariate function, the following theorem is stated and proved in the general setting of moving targets and a multivariate function.
\begin{theorem}\label{thm4}
    Let $n,m \in \mathbb{N},\ \Psi: \mathbb{Z}^n \setminus\{\mathbf{0}\} \to [0,\frac{1}{2})$ be given, and $\mathbf{y}_\mathbf{q}, \frac{\mathbf{a}_\mathbf{q}}{b_\mathbf{q}}, \tilde{A}_\mathbf{q}$ be defined as in Definition \ref{def}. Suppose that $\sum_\mathbf{q} |\tilde A_\mathbf q| = \infty$ and the sequence $(\tilde A_\mathbf{q})_{\mathbf q \in \mathbb{Z}^n \setminus \{\mathbf 0\}}$ is QIA. Then
    $$\left|\limsup_{|\mathbf q| \to \infty} \tilde A_\mathbf{q} \right| = 1.$$
\end{theorem}
\begin{proof}
    It follows from the following lemma and the proof of \cite[Proposition 1]{AR22}. The argument in \cite[Proposition 1]{AR22} concerns the set $A_\mathbf{q}$, but the following lemma, a modification of \cite[Lemma 7]{AR22}, enables its application to $\tilde{A}_\mathbf{q}$.
\end{proof}

The following lemma generalizes \cite[Lemma 7]{AR22} to $\tilde{A}_\mathbf{q}$, or \cite[Proposition 1]{BHV24} to arbitrary dimensions. The proof follows a similar technique as in \cite{AR23}.
\begin{lemma}
    The same setting as in Theorem \ref{thm4} is used. There exists a constant $C_m>0$ such that for every open set $U \subset [0,1)^{nm}$, every $\mathbf{q} \in \mathbb{Z}^n$ with big enough $|\mathbf q|$ (there exists $Q_U >0 $ such that for all $|\mathbf{q}| \geq Q_U$) satisfies
    $$|\tilde A_\mathbf q \cap U| \geq C_m |\tilde A_\mathbf{q}| |U|.$$
\end{lemma}
\begin{proof}
    From the proof of \cite[Lemma 7]{AR22}, it is enough to consider the case when $U$ is a ball $W$ in $\mathbb R^{nm} /\mathbb Z^{nm} \cong [0,1)^{nm}$ under the maximum norm. Without loss of generality, we may assume that $|\mathbf q|$ appears in the first coordinate of $\mathbf q$. For a subset $B \subset [0,1)^{nm}$ and a matrix $\mathbf{z} \in [0,1)^{(n-1)m}$, we define
    $$B_\mathbf{z} := \left\{ \mathbf x \in [0,1)^m : \begin{pmatrix} \mathbf x \\ \mathbf z \end{pmatrix} \in B\right\}.$$
    Let $Y$ be the image of $W$ by the projection $[0,1)^{nm} \to [0,1)^{(n-1)m}$ dropping the first row. Then for $\mathbf z \in Y$, we have
    \begin{align*}
        (\tilde A_\mathbf q \cap W)_\mathbf z = \left\{ \mathbf x \in W_\mathbf z : \exists \mathbf{p} \in \mathbb{Z}^m, \left| \mathbf q \begin{pmatrix} \mathbf{x} \\ \mathbf{z} \end{pmatrix} - \mathbf p - \mathbf y_\mathbf q \right| < \Psi(\mathbf q) , \ \gcd(b_\mathbf q \mathbf p +\mathbf a_\mathbf q, \mathbf q) = 1 \right\} \\
        = \left\{ \mathbf x \in W_\mathbf z : \exists \mathbf{p} \in \mathbb{Z}^m, \left| |\mathbf q| \mathbf x + \mathbf q \begin{pmatrix} \mathbf{0} \\ \mathbf{z} \end{pmatrix} - \mathbf p - \mathbf y_\mathbf q \right| < \Psi(\mathbf q) , \ \gcd(b_\mathbf q \mathbf p +\mathbf a_\mathbf q, \mathbf q) = 1 \right\}.
    \end{align*}
    Let $r$ be the radius of $W$, and $\frac{W}{2}$ denote the ball with same center of $W$ with radius $\frac{r}{2}$. Assume $|\mathbf{q}| > \frac{1}{r},$ and denote
    $$N_\mathbf{z}:=\#\left\{\mathbf p \in \mathbb Z^m : \frac{\mathbf p}{|\mathbf q|} - \frac{1}{|\mathbf{q}|} \mathbf q \begin{pmatrix} \mathbf{0} \\ \mathbf{z} \end{pmatrix} + \frac{\mathbf y_\mathbf q}{|\mathbf q|} \in \left(\frac{W}{2}\right)_{\!\!\mathbf{z}} , \ \gcd(b_\mathbf q \mathbf p +\mathbf a_\mathbf q, \mathbf q) = 1 \right\}.$$
    Since $(\tilde{A}_\mathbf{q})_\mathbf{z}$ is a union of small balls centered at $\frac{\mathbf p}{|\mathbf q|} - \frac{1}{|\mathbf{q}|} \mathbf q {\scriptstyle \begin{pmatrix} \mathbf{0} \\ \mathbf{z} \end{pmatrix}} + \frac{\mathbf y_\mathbf q}{|\mathbf q|}$ with radius $\frac{\Psi(\mathbf{q})}{|\mathbf{q}|}$ (which is smaller than $\frac{r}{2}$), at least $N_\mathbf{z}$ of these balls are contained in $W_\mathbf{z}$, hence $|(\tilde A_\mathbf q \cap W)_\mathbf z| \geq N_\mathbf{z} \left(\frac{2\Psi(\mathbf q)}{|\mathbf q|}\right)^m.$ We want to find a lower bound of $N_\mathbf{z}$. If we denote 
    $$V_\mathbf{z} := |\mathbf q| \left(\frac{W}{2}\right)_{\!\!\mathbf{z}} + \mathbf q \begin{pmatrix} \mathbf{0} \\ \mathbf{z} \end{pmatrix} -\mathbf y_\mathbf q,$$
    then it follows from (\ref{mobius}) that
        $$N_\mathbf{z} = \#\left\{\mathbf p \in \mathbb Z^m : \mathbf p \in  V_\mathbf{z}, \ \gcd(b_\mathbf q \mathbf p +\mathbf a_\mathbf q, \mathbf q) = 1 \right\} = \sum_{\mathbf{p} \in \mathbb{Z}^m} \mathds{1}_{V_\mathbf{z}} (\mathbf{p}) \!\!\!\! \sum_{l | \gcd(b_\mathbf q \mathbf p +\mathbf a_\mathbf q, \mathbf q)} \!\!\!\! \mu(l).$$
    Using the same argument as in the proof of Lemma \ref{measure}, if $l | \gcd(b_\mathbf{q} \mathbf{p} + \mathbf{a}_\mathbf{q}, \mathbf{q}),$ then $\gcd(l, b_\mathbf{q}) = 1$, and so there is an inverse $b_\mathbf{q}^{(l)} \in \mathbb{N}$ of $b_\mathbf{q}$ modulo $l$. Hence, $l|b_\mathbf{q}\mathbf{p}+\mathbf{a}_\mathbf{q}$ if and only if $\mathbf{p} \in l\mathbb{Z}^m - b_\mathbf{q}^{(l)}\mathbf{a}_\mathbf{q},$ and thus
    \begin{align*}
        N_\mathbf{z} &= \sum_{l | \mathbf{q}} \mu(l) \sum_{\substack{\mathbf{p} \in \mathbb{Z}^m \\ l| b_\mathbf{q} \mathbf{p} + \mathbf{a}_\mathbf{q}}} \!\!\! \mathds{1}_{V_\mathbf{z}}(\mathbf{p}) = \!\!\!\! \sum_{\substack{l|\mathbf{q} \\ \gcd(l,b_\mathbf{q})=1}} \!\!\!\! \mu(l) \!\!\!\! \sum_{\mathbf{p} \in l\mathbb{Z}^m - b_\mathbf{q}^{(l)}\mathbf{a}_\mathbf{q}} \!\!\!\! \mathds{1}_{V_\mathbf{z}}(\mathbf{p}) \\
        &= \sum_{l| \mathbf{q}, \ \gcd(l, b_\mathbf{q}) =1} \mu(l) \cdot \#\left(l\mathbb{Z}^m \cap (V_\mathbf{z} + b_\mathbf{q}^{(l)}\mathbf{a}_\mathbf{q})\right).
    \end{align*}
    We have a positive constant $k_m$ such that for any $m$-dimensional ball $B$ with side length $L$ and a positive integer $q \in \mathbb{N}$, $\left|\#(q\mathbb{Z}^m \cap B) - (\frac{L}{q})^m\right| \leq k_m\left((\frac{L}{q})^{m-1}+1\right)$. Since $V_\mathbf{z} + b_\mathbf{q}^{(l)}\mathbf{a}_\mathbf{q}$ is a ball with side length $|\mathbf{q}|r,$ it follows from (\ref{mobius}) that
    \begin{align*}
        N_\mathbf{Z} &\geq \sum_{l| \mathbf{q}, \ \gcd(l, b_\mathbf{q}) =1} \mu(l) \left(\frac{|\mathbf{q}|r}{l}\right)^m - \sum_{l| \mathbf{q}, \ \gcd(l, b_\mathbf{q}) =1} k_m \left(\left(\frac{|\mathbf{q}|r}{l}\right)^{m-1} +1 \right)\\
        &= (|\mathbf{q}| r)^m \prod_{p | \mathbf q, p \nmid b_\mathbf q} \left(1-\frac{1}{p^m} \right) - k_m (|\mathbf{q}| r)^{m-1} \prod_{p | \mathbf q, p \nmid b_\mathbf q} \left(1+\frac{1}{p^{m-1}} \right) -k_m 2^t,
    \end{align*}
    where $t$ is the number of primes $p$ such that $p |\mathbf{q}$ and $p\nmid b_\mathbf{q}$. Therefore,
    \begin{align}
        \left|\tilde A_\mathbf{q} \cap W\right| &= \int_Y \left|\left(\tilde A_\mathbf q \cap W\right)_\mathbf z\right| d\mathbf{z} \nonumber \\
        &\geq (2r)^{(n-1)m} \times \inf_{\mathbf{z} \in Y} N_\mathbf{z} \times \left(\frac{2\Psi(\mathbf q)}{|\mathbf q|}\right)^m \nonumber \\
        &\geq (2r)^{nm} \Psi(\mathbf{q})^m \prod_{p | \mathbf q, p \nmid b_\mathbf q} \left(1-\frac{1}{p^m} \right) \label{1st}\\
        &\qquad - \frac{k_m (2r)^{nm} \Psi(\mathbf q)^m}{|\mathbf{q}|r} \prod_{p | \mathbf q, p \nmid b_\mathbf q} \left(1+\frac{1}{p^{m-1}} \right) \label{2nd}\\
        &\qquad - k_m 2^t (2r)^{(n-1)m} \left(\frac{2\Psi(\mathbf q)}{|\mathbf q|}\right)^m. \label{3rd}
    \end{align}
    We observe that (\ref{1st}) equals to $\frac{1}{2^m}|W||\tilde A_\mathbf q|$ from $|W|=(2r)^{nm}$ and Lemma \ref{measure}. Moreover, as $|\mathbf{q}| \to  \infty$, we have
    \begin{align*}
        \frac{(\ref{2nd})}{(\ref{1st})} &= \frac{\frac{k_m}{r|\mathbf q|} \prod_{p | \mathbf q, p \nmid b_\mathbf q} \left(1+\frac{1}{p^{m-1}} \right)}{\prod_{p | \mathbf q, p \nmid b_\mathbf q} \left(1-\frac{1}{p^m} \right)} = \frac{k_m}{r} \frac{1}{|\mathbf q|} \prod_{p | \mathbf q, p \nmid b_\mathbf q} \frac{1+ \frac{1}{p^{m-1}}}{1-\frac{1}{p^m}} \\
        &\leq \frac{k_m}{r} \frac{1}{|\mathbf q|} \prod_{p | \mathbf q, p \nmid b_\mathbf q} \frac{2}{1-\frac{1}{p}} \leq \frac{k_m}{r} \frac{1}{|\mathbf q|} \prod_{p \mid |\mathbf q|} \frac{2}{1-\frac{1}{p}} \to 0,
    \end{align*}
    where the last convergence follows from the relations $|\mathbf{q}|\prod_{p| |\mathbf{q}|} (1-p^{-1}) = \varphi(|\mathbf{q}|)$, $\prod_{p| |\mathbf{q}|} 2 \leq \tau(|\mathbf{q}|)$ and $\frac{\tau(|\mathbf{q}|)}{\varphi(|\mathbf{q}|)} \to 0$ as $|\mathbf{q}| \to \infty$ (see \cite[Thm 315 and 327]{HW08}). A similar calculation yields
    \begin{align*}
        \frac{(\ref{3rd})}{(\ref{1st})} &= \frac{k_m 2^t}{r^m |\mathbf{q}|^m \prod_{p | \mathbf q, p \nmid b_\mathbf q} \left(1-\frac{1}{p^m} \right)} = \frac{k_m}{r^m |\mathbf{q}|^m} \prod_{p | \mathbf q, p \nmid b_\mathbf q} \frac{2}{1-\frac{1}{p^m}} \\
        &\leq \frac{k_m}{r^m |\mathbf{q}|^m} \prod_{p \mid |\mathbf q|} \frac{2}{1-\frac{1}{p}}\to 0.
    \end{align*}
    Therefore, if we let $C_m < \frac{1}{2^m}$, then we conclude that $|\tilde{A}_\mathbf{q} \cap W| \geq C_m |\tilde{A}_\mathbf{q}||W|$ for sufficiently large $|\mathbf{q}|$.
\end{proof}
\begin{proof}[Proof of Theorem \ref{thm1}]
    The convergence part is a direct corollary of the Borel--Cantelli lemma, so it suffices to prove the divergence part. For the inhomogeneous parameter $y \in \mathbb{R}$, we define the sequence $(a_\mathbf{q}, b_\mathbf{q})_{\mathbf{q} \in \mathbb{Z}^2 \setminus \{\mathbf{0}\}}$ as in (\ref{a,b,q}). We assume that $\psi(q) \leq q^{-1}$ for all $q \in \mathbb{N}$, as stated in the first paragraph of this section. Thus, by Theorem \ref{thm3}, the sequence $(\tilde{A}_\mathbf{q})_{\mathbf{q} \in \mathbb{Z}^2 \setminus \{\mathbf{0}\}}$ is QIA. By Lemma \ref{comparable}, we have $\sum_{\mathbf{q} \in \mathbb{Z}^2 \setminus \{\mathbf{0}\}} |\tilde{A}_\mathbf{q}| = \infty.$ Therefore, by Theorem \ref{thm4}, $\limsup_{|\mathbf{q}| \to \infty} \tilde{A}_\mathbf{q}$ has full measure. Since $\limsup_{|\mathbf{q}| \to \infty} \tilde{A}_\mathbf{q} \subseteq \mathscr{A}_{2,1}^y(\psi)$, the theorem follows.
\end{proof}

\section{Proof of Theorem \ref{thm2}}
\subsection{Quasi-independence on average}
In this section, we consider the case where the inhomogeneous parameter $\mathbf{y} \in \mathbb{Q}^m$ is rational. Let $\mathbf{y} = \frac{\mathbf{a}}{b}$ with $\mathbf{a} \in \mathbb{Z}^m, b\in \mathbb{N}$ satisfying $\gcd(\mathbf{a},b) = 1$. We then fix the sequence $(\mathbf{a}_\mathbf{q}, b_\mathbf{q}) = (\mathbf{a},b)$, which arises in Definition \ref{def}. For this case, we especially denote $A''_{n,m}(\mathbf{q}, \delta) = \tilde{A}_{n,m}(\mathbf{q},\delta)$ as in Definition \ref{def}, that is,
$$A''_{n,m}(\mathbf{q},\delta) = \{\mathbf{x} \in [0,1)^{nm} : |\mathbf{qx} - \mathbf{p} - \mathbf{y}| < \delta, \ \gcd(\mathbf{q}, b\mathbf{p} + \mathbf{a}) = 1\}.$$

Let $\psi: \mathbb{N} \to \mathbb{R}_{\geq 0}$ be a univariate approximating function . If we denote $\mathscr{A}''^{\mathbf{y}}_{n,m}(\psi) =\limsup_{|\mathbf{q}| \to \infty} A''_{n,m}(\mathbf{q}, \psi(|\mathbf{q}|))$, then we have $\mathscr{A}''^{\mathbf{y}}_{n,m} \subseteq \mathscr{A}^{\mathbf{y}}_{n,m}.$ We will show that $(A''_{n,m}(\mathbf{q}, \psi(|\mathbf{q}|)))_{\mathbf{q} \in \mathbb{Z}^m \setminus \{\mathbf{0}\}}$ is quasi-independent on average in the divergent case when $nm\geq 2$.

\begin{lemma}[\cite{Gal65}, Lemma 1]\label{gallagher}
    Let $U(\delta) = (-\delta,\delta)^m \subset \mathbb{R}^m$. Then for $a,b \geq 0$,
    $$\sum_{\mathbf{j} \in \mathbb{Z}^m \setminus \{\mathbf{0}\}} |U(a) \cap (U(b)+\mathbf{j})| \ll (ab)^m.$$
\end{lemma}
\begin{lemma}\label{rationalqia}
    Let $n=1, m \in \mathbb{N}.$ For $q \in \mathbb{N} , r \in \mathbb{Z}\setminus \{0\}$ with $|r|<q$ and $\delta_1, \delta_2 \geq 0$,
    $$|A''_{1,m}(q, \delta_1)\cap A''_{1,m}(r,\delta_2)| \ll (b\delta_1\delta_2)^m.$$
\end{lemma}
\begin{proof}
    We modify the calculation of \cite{Gal65}. Consider $A_{1,m}(r,\delta_2)$ containing $A''_{1,m}(r,\delta_2)$.
    \begin{align*}
        |&A''_{1,m}(q, \delta_1) \cap A_{1,m}(r,\delta_2)| \\
        &= \sum_{\substack{\mathbf{p} \in \{1, \cdots, q\}^m \\ \gcd(q, b\mathbf{p} + \mathbf{a}) = 1}} \sum_{\mathbf{s} \in \mathbb{Z}^m} \left|\left\{U\left(\frac{\delta_1}{q}\right) + \frac{\mathbf{p} + \mathbf{y}}{q}\right\} \cap \left\{U\left(\frac{\delta_2}{|r|}\right) + \frac{\mathbf{s} + \mathbf{y}}{r}\right\}\right| \\
        &= \frac{\gcd(q,r)^m}{(bq|r|)^m} \sum_{\substack{\mathbf{p} \in \{1, \cdots, q\}^m \\ \gcd(q, b\mathbf{p} + \mathbf{a}) = 1}} \sum_{\mathbf{s} \in \mathbb{Z}^m} \left|\left\{ U\left(\frac{b|r|\delta_1}{\gcd(q,r)}\right) + \frac{br(\mathbf{p} + \mathbf{y})}{\gcd(q,r)} \right\} \right.\\ 
        &\qquad \qquad \qquad \qquad \qquad \qquad \qquad \qquad \cap \left. \left\{U\left(\frac{bq\delta_2}{\gcd(q,r)}\right) + \frac{bq(\mathbf{s} + \mathbf{y})}{\gcd(q,r)}\right\}\right| \\
        &= \frac{\gcd(q,r)^m}{(bq|r|)^m} \sum_{\substack{\mathbf{p} \in \{1, \cdots, q\}^m \\ \gcd(q, b\mathbf{p} + \mathbf{a}) = 1}} \sum_{\mathbf{s} \in \mathbb{Z}^m} \left|\left\{ U\left(\frac{b|r|\delta_1}{\gcd(q,r)}\right)\right\} \right.\\ 
        &\qquad \qquad \qquad \qquad \qquad \qquad \cap \left. \left\{U\left(\frac{bq\delta_2}{\gcd(q,r)}\right) + \frac{bq(\mathbf{s} + \mathbf{y}) - br(\mathbf{p} + \mathbf{y})}{\gcd(q,r)}\right\}\right| \\
        &\leq \frac{\gcd(q,r)^{2m}}{(bq|r|)^m} \sum_{\mathbf{j} \in \mathbb{Z}^m\setminus \{\mathbf{0}\}} \left|U\left(\frac{b|r|\delta_1}{\gcd(q,r)}\right) \cap \left\{U\left(\frac{bq\delta_2}{\gcd(q,r)} + \mathbf{j} \right)\right\}\right|.
    \end{align*}
    Here, $\mathbf{j} \in \mathbb{Z}^m$ in the last inequality can be chosen from the equality
    \begin{equation}\label{choosej}
        bq(\mathbf{s} + \mathbf{y}) - br(\mathbf{p} + \mathbf{y}) = q(b\mathbf{s} + \mathbf{a}) - r(b\mathbf{p} + \mathbf{a}) =\gcd(q,r) \mathbf{j},
    \end{equation}
    and $\mathbf{j} \neq \mathbf{0}$ since $\frac{b\mathbf{p} + \mathbf{a}}{q} \neq \frac{b\mathbf{s} + \mathbf{a}}{r}$ by the fact that $\frac{b\mathbf{p} + \mathbf{a}}{q}$ is a reduced fraction with a larger denominator of that of $\frac{b\mathbf{s} + \mathbf{a}}{r}$. The number of pairs $(\mathbf{p}, \mathbf{s})$ satisfying (\ref{choosej}) is at most $\gcd(q,r)^m$, because if (\ref{choosej}) is satisfied by another $(\mathbf{p}',\mathbf{s}')$, then $q(\mathbf{s} - \mathbf{s}') = r(\mathbf{p} - \mathbf{p}')$, so $\mathbf{p} \equiv \mathbf{p}' \pmod{\frac{q}{\gcd(q,r)}}$ where $\mathbf{p}, \mathbf{p}' \in \{1, \cdots, q\}^m$. Therefore by Lemma \ref{gallagher},
    $$|A''_{1,m}(q,\delta_1) \cap A_{1,m}(r,\delta_2)| \ll \left(\frac{\gcd(q,r)^2}{bq|r|} \frac{b|r|\delta_1}{\gcd(q,r)} \frac{bq\delta_2}{\gcd(q,r)}\right)^m = (b\delta_1\delta_2)^m.$$
\end{proof}
\begin{theorem}\label{thm5}
    Let $nm \geq 2,\ \psi: \mathbb{N} \to [0,\frac{1}{2})$, and let $\mathbf{y} = \frac{\mathbf{a}}{b} \in \mathbb{Q}^m$ be a rational inhomogeneous parameter with $\mathbf{a} \in \mathbb{Z}^m, b \in \mathbb{N}$ and $\gcd(\mathbf{a},b)=1$. If $\sum_{q=1}^\infty q^{n-1}\psi(q)^m = \infty$, then the sequence $(A''_{n,m}(\mathbf{q}, \psi(|\mathbf{q}|)))_{\mathbf{q}\in\mathbb{Z}^n \setminus \{\mathbf{0}\}}$ is quasi-independent on average.
\end{theorem}
\begin{proof}
    We use a similar logic of Step 1 in the proof of Theorem \ref{thm3}. By Lemma \ref{comparable} and \ref{notparallel}, it is enough to show that
    $$\sum_{1\leq |\mathbf{q}| \leq Q} \sum_{\substack{1\leq |\mathbf{r}| < |\mathbf{q}| \\ \mathbf{r}\parallel \mathbf{q}}} \left|A''_{n,m}(\mathbf{q}, \psi(|\mathbf{q}|)) \cap A''_{n,m}(\mathbf{r}, \psi(|\mathbf{r}|))\right| \ll \left(\sum_{1\leq |\mathbf{q}| \leq Q} |A_{n,m}(\mathbf{q}, \psi(|\mathbf{q}|))| \right)^2.$$
    For $\mathbf{q}, \mathbf{r} \in \mathbb{Z}^n \setminus \{\mathbf{0}\}$ satisfying $\mathbf{q} \parallel \mathbf{r}$ and $1\leq |\mathbf{r}| < |\mathbf{q}|$, we have a primitive vector $\mathbf{k} \in \mathbb{Z}^n \setminus \{\mathbf{0}\}$ such that $\mathbf{q} = q\mathbf{k}, \mathbf{r} = r\mathbf{k}$, where $q=\gcd(\mathbf{q}), r=\pm \gcd(\mathbf{r})$. Since $(\mathbf{a}, b)$ is fixed, by Lemma \ref{intersection}, we have
    $$|A''_{n,m}(\mathbf{q}, \psi(|\mathbf{q}|)) \cap A''_{n,m}(\mathbf{r}, \psi(|\mathbf{r}|))| = |A''_{1,m}(q, \psi(|\mathbf{q}|)) \cap A''_{1,m}(r, \psi(|\mathbf{r}|))|.$$
    Therefore by Lemma \ref{rationalqia}, we have
    $$|A''_{n,m}(\mathbf{q}, \psi(|\mathbf{q}|)) \cap A''_{n,m}(\mathbf{r}, \psi(|\mathbf{r}|))| \ll \left(b\psi(|\mathbf{q}|)\psi(|\mathbf{r}|)\right)^m,$$
    where the right hand side is $(\frac{b}{4})^m |A_{n,m}(\mathbf{q}, \psi(|\mathbf{q}|))||A_{n,m}(\mathbf{r},\psi(|\mathbf{r}|))|.$ Since $b$ is fixed, we can conclude that the sequence $(A''_{n,m}(\mathbf{q}, \psi(|\mathbf{q}|)))_{\mathbf{q} \in \mathbb{Z}^n \setminus \{\mathbf{0}\}}$ is QIA.
\end{proof}
\begin{proof}[First proof of Theorem \ref{thm2}]
    We only need to prove the divergence part and consider the case where $\psi: \mathbb{N} \to [0,\frac{1}{2})$ as noted in the remark following Definition \ref{def}. By Theorem \ref{thm5}, the sequence $(A''_{n,m}(\mathbf{q}, \psi(|\mathbf{q}|)))_{\mathbf{q} \in \mathbb{Z}^n \setminus \{\mathbf{0}\}}$ is QIA. Since $A''$ is a special case of $\tilde A$, Theorem \ref{thm4} implies that $\mathscr{A}''^{\mathbf{y}}_{n,m}(\psi)$ has full measure. Therefore, we have $|\mathscr{A}^{\mathbf{y}}_{n,m}(\psi)| \geq |\mathscr{A}''^{\mathbf{y}}_{n,m}(\psi)| = 1.$
\end{proof}

\subsection{Zero-one law}
Since we proved Theorem \ref{thm2}, we know that the set $\mathscr{A}^{\mathbf{y}}_{n,m}(\psi)$ has measure either zero or one for $nm\geq 2$ and $\mathbf{y} \in \mathbb{Q}^m$. In this section, we provide an a priori proof of the zero-one law for cases with a rational inhomogeneous parameter, extending this result to all $n,m \in \mathbb{N}$. This result is of independent interest and also serves as an alternative proof of Theorem \ref{thm2}.

Beresnevich and Velani \cite{BV08} generalized the zero-one law originally formulated by Cassels \cite{Cas50} and Gallagher \cite{Gal61}. Their result \cite{BV08} states that for the homogeneous case with any $n,m \geq 1$, and a multivariate approximating function $\Psi: \mathbb{Z}^n \setminus \{\mathbf 0\} \to \mathbb R_{\geq 0}$, we have $|\mathscr A _{n,m} (\Psi)| \in \{0,1\}.$ The core idea of their proof in the homogeneous case relies on the ergodicity of the map $\mathbf{x} \mapsto 2\mathbf{x} \pmod 1$, and an invariant set under this map. The rational nature of the inhomogeneous parameter allows us to extend this approach.
\begin{theorem}\label{thm6}
    Let $n,m \in \mathbb{N}$, and let $\mathbf{y} = \frac{\mathbf{a}}{b} \in \mathbb{Q}^m$ be an inhomogeneous parameter with $\mathbf{a} \in \mathbb{Z}^m, b\in \mathbb{N}$ satisfying $\gcd(\mathbf{a}, b)=1$. For any multivariate approximating function $\Psi: \mathbb{Z}^n \setminus \{\mathbf 0\} \to \mathbb R_{\geq 0}$, we have
    $$|\mathscr{A}^\mathbf{y}_{n,m}(\Psi)| \in \{0,1\}.$$
\end{theorem}
\begin{proof}
    We may assume that $\frac{\Psi(\mathbf{q})}{|\mathbf q|} \to 0$ as $|\mathbf{q}| \to \infty$ (see (9) of \cite{BV08}). Define 
    $$\mathscr{F}^\mathbf{y}_{n,m}(\Psi) = \bigcup_{k=1}^\infty \mathscr{A}^\mathbf{y}_{n,m}(k\Psi).$$
    Since \cite[Lemma 4]{BV08}, which can be traced back to \cite[Lemma 9]{Cas50}, is valid for the inhomogeneous case as well, it follows that
    $$|\mathscr{F}^\mathbf{y}_{n,m}(\Psi)| = |\mathscr{A}^\mathbf{y}_{n,m}(\Psi)|.$$
    Let $l=b+1$ and $T_l : [0,1)^{nm} \to [0,1)^{nm}$ be a function sending $\mathbf{x} \mapsto l\mathbf{x} \mod 1$. Then we have $T_l(\mathscr{F}^\mathbf{y}_{n,m}(\Psi)) \subseteq \mathscr{F}^\mathbf{y}_{n,m}(\Psi)$ since
    \begin{align*}
        \mathbf{x} \in \mathscr{F}^\mathbf{y}_{n,m}(\Psi) &\Leftrightarrow \exists k\in \mathbb{N} : \left|\mathbf{q}\mathbf{x} - \mathbf{p} - \frac{\mathbf{a}}{b}\right| < k\Psi(\mathbf{q}) \text{ for inf. many } (\mathbf{q}, \mathbf{p}) \\
        &\Leftrightarrow \exists k\in \mathbb{N} : \left|\mathbf{q}l\mathbf{x} - l\mathbf{p} - l\frac{\mathbf{a}}{b}\right| < kl\Psi(\mathbf{q}) \text{ for inf. many } (\mathbf{q}, \mathbf{p}) \\
        &\Leftrightarrow \exists k \in \mathbb{N} : \left|\mathbf{q}l\mathbf{x} - l\mathbf{p} - \mathbf{a} - \frac{\mathbf{a}}{b}\right| < kl\Psi(\mathbf{q}) \text{ for inf. many } (\mathbf{q}, \mathbf{p}) \\
        &\Rightarrow \exists k' \in \mathbb{N} : \left|\mathbf{q}T_l(\mathbf{x}) - \mathbf{p}'- \frac{\mathbf{a}}{b}\right| < k'\Psi(\mathbf{q}) \text{ for inf. many } (\mathbf{q}, \mathbf{p}') \\
        &\Rightarrow T_l(\mathbf{x}) \in \mathscr{F}^\mathbf{y}_{n,m}(\Psi).
    \end{align*}
    Therefore by \cite[Lemma 2]{BV08}, $|\mathscr{F}^\mathbf{y}_{n,m}(\Psi)| \in \{0,1\}$, thus $|\mathscr{A}^\mathbf{y}_{n,m}(\Psi)| \in \{0,1\}$.
\end{proof}
\begin{proof}[Second proof of Theorem \ref{thm2}]
    We only need to prove the divergence part. By Theorem \ref{thm5} and the divergence Borel--Cantelli lemma, we have $|\mathscr{A}^{\mathbf{y}}_{n,m} (\psi)| \geq |\mathscr{A}''^{\mathbf{y}}_{n,m} (\psi)| > 0.$ Therefore, $|\mathscr{A}^{\mathbf{y}}_{n,m} (\psi)| = 1$ by the zero-one law in Theorem \ref{thm6}.
\end{proof}

\end{document}